\documentclass[twoside,12pt]{article}
\usepackage{amsmath} 
\usepackage{amssymb} 
\usepackage{amsthm}  
\usepackage{amscd}
\usepackage{mathrsfs}
\usepackage{latexsym}
\usepackage{upgreek}
\usepackage{upgreek}
\usepackage{graphicx}
\usepackage{array}
\usepackage[utf8]{inputenc}
\usepackage{tikz}
\usepackage{dcolumn}
\usepackage{multirow}
\usepackage{booktabs}
\usepackage{lscape}
\usepackage{verbatim}
\usepackage{pgfplots}
\DeclareMathAlphabet{\mathpzc}{OT1}{pzc}{m}{it}

\newtheorem{theorem}{Theorem}[section]

\newcommand{\abs}[1]{\lvert#1\rvert}

\begin{document}
\title{A new proof of an equality associated with submatrices and eigenvector elements}
\author{Liguo He\thanks{Corresponding author}, \quad Guirong Song \thanks{E-mail address:  helg-lxy@sut.edu.cn(L.He), songgr2004.163.com(G. Song)},
\\{\footnotesize Dept. of Math., Shenyang University of Technology, Shenyang, 110870,  PR China}}
\date{}

\maketitle

\noindent \textbf{Abstract.} {\footnotesize{By using the methods of Cauchy-Binet type formula and adjugate matrix respectively, a wonderful equality relating to the elements of eigenvectors, the eigenvalues and the submatrix eigenvalues is proved in arXiv:1908.03795. In the note, we use matrix block to provide a new and shorter proof for the equality. }}

Let $A$ be a Hermite matrix with eigenvalues $\lambda_i(A)$ and normed eigenvectors $v_i$, whose eigenvector elements are written as $v_{i,j}$. Let $M_j$ be an $(n-1)\times (n-1)$ submatrix which is derived by deleting the jth row elements and the jth column elements, and its eigenvalues are written as $\lambda_k(M_j)$. As the main result in \cite{T13}, Tao et al. prove twice the following result applying the methods of Cauchy-Binet type formula and adjugate matrix, respectively.

\begin{theorem}
Let $A$ be an $n\times n$ Hermite matrix, then the following equality holds  $$ \abs{v_{i,j}}^2 \Pi_{k=1, k\neq i}^n (\lambda_i(A) - \lambda_k(A)) = \Pi_{k=1}^{n-1}(\lambda_i(A)-\lambda_k(M_j)).$$
\end{theorem}
\begin{proof}
See \cite[Lemma 2]{T13}.
\end{proof}

Because there are the same eigenvalues between the similar matrices, as pointed out in \cite{T13}, the above equality proofs regarding $A$ and $A - \lambda_i(A)E_n$ are equivalent, and so we may take $\lambda_i(A)$ to be zero, further choose $i = n, j=1$, and thus obtain the following reduction form

$$ \abs{v_{n,1}}^2\Pi_{k=1}^{n-1}\lambda_k(A) = \Pi_{k=1}^{n-1}\lambda_k(M_1).$$

In this note, we actually take $i = n, j = n$ and prove the following reduction form

$$ \abs{v_{n,n}}^2\Pi_{k=1}^{n-1}\lambda_k(A) = \Pi_{k=1}^{n-1}\lambda_k(M_n).$$

Since $M_n$ is also a Hermite matrix, we have $\abs{M_n} = \Pi_{k=1}^{n-1}\lambda_k(M_n)$.

\begin{theorem}
Let $A$ be an $n\times n$ Hermite matrix with $\lambda_n = 0$, then it is true that $\abs{M_n} = \abs{\Lambda_n}\abs{u_{n,n}}^2$.
\end{theorem}
\begin{proof}

Since  $A$ be an $n\times n$ Hermite matrix with $\lambda_n = 0$, there exists an $n\times n$ unitary matrix $U$ such that $U^{-1}AU = \Lambda$ and $\Lambda$ is a diagonal matrix with $\lambda_n=0$. Thus we may suppose that

 \[ A = (a_{i,j}) = \left( \begin{array}{cc}
M_n & B  \\
\bar{B}^{'}& a_{n,n}
\end{array} \right), \,  U =(u_{i,j}) = \left( \begin{array}{cc}
U_n & C_1  \\
C_2& u_{n,n}\\
\end{array}\right),\,  U^{-1} = \left( \begin{array}{cc}
\bar{U}^{'}_n & \bar{C}^{'}_2  \\
\bar{C}^{'}_1& \bar{u}_{n,n}\\
\end{array}\right),\]

\[ \Lambda = \left(
\begin{array}{cc}
\Lambda_n&\\
&0\\
\end{array}\right), \quad \Lambda_n = \left(
\begin{array}{cccc}
\lambda_1 &&&\\
 &\lambda_2&&\\
&&\ddots&\\
&&&\lambda_{n-1}\\
\end{array}\right).\]

We may deduce that
\[A = U\Lambda U^{-1} = \left( \begin{array}{cc}
U_n & C_1  \\
C_2& u_{n,n}\\
\end{array}\right)\left(
\begin{array}{cc}
\Lambda_n&\\
&0\\\end{array}\right) \left( \begin{array}{cc}
\bar{U}^{'}_n & \bar{C}^{'}_2  \\
\bar{C}^{'}_1& \bar{u}_{n,n}\\
\end{array}\right)= \left( \begin{array}{cc}
U_n\Lambda_n\bar{U}^{'}_n & U_n\Lambda_n\bar{C}^{'}_2\\
C_2\Lambda_n\bar{U}^{'}_n& C_2\Lambda_n\bar{C}^{'}_2\\
\end{array}\right).\]

It follows $M_n = U_n\Lambda_n\bar{U}^{'}_n$, and so $\abs{M_n} = \abs{\Lambda_n}\abs{U_n\bar{U}^{'}_n}$. Hence it suffices to prove $\abs{U_n\bar{U}^{'}_n}=\abs{u_{n,n}}^2$, where $\abs{u_{n,n}}^2 = u_n\bar{u}_n$.
Because $U$ is a unitary matrix, we get that

\[ E_n = U\bar{U}^{'} = \left( \begin{array}{cc}
U_n & C_1  \\
C_2& u_{n,n}\\
\end{array}\right)\left( \begin{array}{cc}
\bar{U}^{'}_n & \bar{C}^{'}_2  \\
\bar{C}^{'}_1& \bar{u}_{n,n}\\
\end{array}\right)=\left( \begin{array}{cc}
U_n\bar{U}^{'}_n + C_1\bar{C}^{'}_1 &U_n\bar{C}^{'}_2+ C_1\bar{u}_{n,n}\\
C_2\bar{U}^{'}_n+u_{n,n}\bar{C}^{'}_1& C_2\bar{C}^{'}_2+ u_{n,n}\bar{u}_{n,n}\\
\end{array}\right),\]
thus $\, U_n\bar{U}^{'}_n + C_1\bar{C}^{'}_1 = E_{n-1}$. We derive that $$\abs{U_n\bar{U}^{'}_n} = \abs{E_{n-1}-C_1\bar{C}^{'}_1} = \abs{E_1 - \bar{C}^{'}_1C_1} = \abs{u_{n,n}}^2.$$ Note that $\bar{C}^{'}_1C_1 + \bar{u}_{n,n}u_{n,n} = \sum_{i = 1}^n \bar{u}_{i,n}u_{i,n} = 1$. The proof is complete.
\end{proof}

\end{document}